\documentclass[a4paper,twoside]{article}
\usepackage{a4}
\usepackage{amssymb}
\usepackage{amsmath}
\usepackage{upref}
\usepackage[active]{srcltx}
\usepackage[pagebackref,colorlinks,citecolor=blue,linkcolor=blue]{hyperref}
\allowdisplaybreaks[2] 
\newcount\minutes \newcount\hours
\hours=\time \divide\hours 60 \minutes=\hours
 \multiply\minutes-60
\advance\minutes \time
\newcommand{\klockan}{\the\hours:{\ifnum\minutes<10 0\fi}\the\minutes}
\newcommand{\tid}{\today\ \klockan}
\newcommand{\prtid}{\smash{\raise 10mm \hbox{\LaTeX ed \tid}}}
\renewcommand{\prtid}{}
\makeatletter  \headheight 10pt
\def\sectionmark#1{} 
\def\subsectionmark#1{}
\newcommand{\sectnr}{\ifnum \c@secnumdepth >\z@
                 \thesection.\hskip 1em\relax \fi}
\def\@evenhead{\footnotesize\rm\thepage\hfil\leftmark\hfil\llap{\prtid}}
\def\@oddhead{\footnotesize\rm\rlap{\prtid}\hfil\rightmark\hfil\thepage}
\def\tableofcontents{\section*{Contents} 
 \@starttoc{toc}}
\makeatother
\makeatletter
\def\@biblabel#1{#1.}
\makeatother
\makeatletter
\let\Thebibliography=\thebibliography
\renewcommand{\thebibliography}[1]{\def\@mkboth##1##2{}\Thebibliography{#1}
\addcontentsline{toc}{section}{References}
\frenchspacing 
\setlength{\@topsep}{0pt}
\setlength{\itemsep}{0pt}%
\setlength{\parskip}{0pt plus 2pt}%
} \makeatother
\makeatletter
\def\mdots@{\mathinner.\nonscript\!.%
 \ifx\next,.\else\ifx\next;.\else\ifx\next..\else
 \nonscript\!\mathinner.\fi\fi\fi}
\let\ldots\mdots@
\let\cdots\mdots@
\let\dotso\mdots@
\let\dotsb\mdots@
\let\dotsm\mdots@
\let\dotsc\mdots@
\def\vdots{\vbox{\baselineskip2.8\p@ \lineskiplimit\z@
    \kern6\p@\hbox{.}\hbox{.}\hbox{.}\kern3\p@}}
\def\ddots{\mathinner{\mkern1mu\raise8.6\p@\vbox{\kern7\p@\hbox{.}}%
    \raise5.8\p@\hbox{.}\raise3\p@\hbox{.}\mkern1mu}}
\makeatother
\makeatletter
\let\Enumerate=\enumerate
\renewcommand{\enumerate}{\Enumerate%
\setlength{\@topsep}{0pt}
\setlength{\itemsep}{0pt}%
\setlength{\parskip}{0pt plus 1pt}%
\renewcommand{\theenumi}{\textup{(\alph{enumi})}}%
\renewcommand{\labelenumi}{\theenumi}%
}
\let\endEnumerate=\endenumerate
\renewcommand{\endenumerate}{\endEnumerate\unskip}
\makeatother
\makeatletter
\def\@seccntformat#1{\csname the#1\endcsname.\quad}
\makeatother
\usepackage{ifthen}
\newcommand{\authortitle}[3]{\author{#1}\markboth{#1}{#2}\ifthenelse{\equal{#3}{}}{\title{#2}}{\title{#3}}}
\newcommand{\art}[6]{{\sc #1, \rm #2, \it #3 \bf #4 \rm (#5), \mbox{#6}.}}
\newcommand{\artnopt}[6]{{\sc #1, \rm #2, \it #3 \bf #4 \rm (#5), \mbox{#6}}}
\newcommand{\auth}[2]{{#1, #2.}}
\newcommand{\artin}[3]{{\sc #1, \rm #2,  in #3.}}

\newcommand{\arttoappear}[3]{{\sc #1, \rm #2, to appear in \it #3}}
\newcommand{\book}[3]{{\sc #1, \it #2, \rm #3.}}
\newcommand{\AND}{{\rm and }}
\RequirePackage{amsthm}
\newtheoremstyle{descriptive}%
  {\topsep}   
  {\topsep}   
  {\rmfamily} 
  {}          
  {\bfseries} 
  {.}         
  { }         
  {}          
\newtheoremstyle{propositional}%
  {\topsep}   
  {\topsep}   
  {\itshape}  
  {}          
  {\bfseries} 
  {.}         
  { }         
  {}          
\theoremstyle{propositional}
\newtheorem{thm}{Theorem}[section]
\newtheorem{theorem}[thm]{Theorem}  
\newtheorem{proposition}[thm]{Proposition}
\newtheorem{lemma}[thm]{Lemma} 
\newtheorem{corollary}[thm]{Corollary} 
\theoremstyle{descriptive}
\newtheorem{definition}[thm]{Definition}
\makeatletter
\renewenvironment{proof}[1][\proofname]{\par
  \pushQED{\qed}%
  \normalfont
  \trivlist
  \item[\hskip\labelsep
        \itshape
    #1\@addpunct{.}]\ignorespaces
}{%
  \popQED\endtrivlist\@endpefalse
} \makeatother

\newcommand{\setm}{\setminus}
\renewcommand{\emptyset}{\varnothing}
\def\vint{\mathop{\mathchoice%
          {\setbox0\hbox{$\displaystyle\intop$}\kern 0.22\wd0%
           \vcenter{\hrule width 0.6\wd0}\kern -0.82\wd0}%
          {\setbox0\hbox{$\textstyle\intop$}\kern 0.2\wd0%
           \vcenter{\hrule width 0.6\wd0}\kern -0.8\wd0}%
          {\setbox0\hbox{$\scriptstyle\intop$}\kern 0.2\wd0%
           \vcenter{\hrule width 0.6\wd0}\kern -0.8\wd0}%
          {\setbox0\hbox{$\scriptscriptstyle\intop$}\kern 0.2\wd0%
           \vcenter{\hrule width 0.6\wd0}\kern -0.8\wd0}}%
          \mathopen{}\int}
\newcommand{\Cp}{{C_p}}
\newcommand{\bCp}{{\itoverline{C}_p}}
\DeclareMathOperator{\Div}{div}
\DeclareMathOperator{\diam}{diam} 
\DeclareMathOperator{\dist}{dist}
\DeclareMathOperator{\Lip}{Lip}
\newcommand{\bdy}{\partial}
\newcommand{\loc}{_{\rm loc}}
{\catcode`p =12 \catcode`t =12 \gdef\eeaa#1pt{#1}}      
\def\accentadjtext#1{\setbox0\hbox{$#1$}\kern   
                \expandafter\eeaa\the\fontdimen1\textfont1 \ht0 }
\def\accentadjscript#1{\setbox0\hbox{$#1$}\kern 
                \expandafter\eeaa\the\fontdimen1\scriptfont1 \ht0 }
\def\accentadjscriptscript#1{\setbox0\hbox{$#1$}\kern   
                \expandafter\eeaa\the\fontdimen1\scriptscriptfont1 \ht0 }
\def\accentadjtextback#1{\setbox0\hbox{$#1$}\kern       
                -\expandafter\eeaa\the\fontdimen1\textfont1 \ht0 }
\def\accentadjscriptback#1{\setbox0\hbox{$#1$}\kern     
                -\expandafter\eeaa\the\fontdimen1\scriptfont1 \ht0 }
\def\accentadjscriptscriptback#1{\setbox0\hbox{$#1$}\kern 
                -\expandafter\eeaa\the\fontdimen1\scriptscriptfont1 \ht0 }
\def\itoverline#1{{\mathsurround0pt\mathchoice
        {\rlap{$\accentadjtext{\displaystyle #1}
                \accentadjtext{\vrule height1.593pt}
                \overline{\phantom{\displaystyle #1}
                \accentadjtextback{\displaystyle #1}}$}{#1}}
        {\rlap{$\accentadjtext{\textstyle #1}
                \accentadjtext{\vrule height1.593pt}
                \overline{\phantom{\textstyle #1}
                \accentadjtextback{\textstyle #1}}$}{#1}}
        {\rlap{$\accentadjscript{\scriptstyle #1}
                \accentadjscript{\vrule height1.593pt}
                \overline{\phantom{\scriptstyle #1}
                \accentadjscriptback{\scriptstyle #1}}$}{#1}}
        {\rlap{$\accentadjscriptscript{\scriptscriptstyle #1}
                \accentadjscriptscript{\vrule height1.593pt}
                \overline{\phantom{\scriptscriptstyle #1}
                \accentadjscriptscriptback{\scriptscriptstyle #1}}$}{#1}}}}
\def\itunderline#1{{\mathsurround0pt\mathchoice
        {\rlap{$\underline{\phantom{\displaystyle #1}
                \accentadjtextback{\displaystyle #1}}$}{#1}}
        {\rlap{$\underline{\phantom{\textstyle #1}
                \accentadjtextback{\textstyle #1}}$}{#1}}
        {\rlap{$\underline{\phantom{\scriptstyle #1}
                \accentadjscriptback{\scriptstyle #1}}$}{#1}}
        {\rlap{$\underline{\phantom{\scriptscriptstyle #1}
                \accentadjscriptscriptback{\scriptscriptstyle #1}}$}{#1}}}}
\def\cprime{{\mathsurround0pt$'$}}
\newcommand{\dmu}{d\mu}
\newcommand{\ds}{ds}
\newcommand{\bdyOmegaX}{\bdy\Omega\setm\{\infty\}}
\newcommand{\clV}{\overline{V}}

\renewcommand{\phi}{\varphi}
\newcommand{\p}{{$p\mspace{1mu}$}}
\newcommand{\R}{\mathbb{R}}
\newcommand{\eR}{{\overline{\R\kern-0.08em}\kern 0.08em}} 

\newcommand{\limplus}{{\mathchoice{\vcenter{\hbox{$\scriptstyle +$}}}
  {\vcenter{\hbox{$\scriptstyle +$}}}
  {\vcenter{\hbox{$\scriptscriptstyle +$}}}
  {\vcenter{\hbox{$\scriptscriptstyle +$}}}
}}
\newcommand{\limminus}{{\mathchoice{\vcenter{\hbox{$\scriptstyle -$}}}
  {\vcenter{\hbox{$\scriptstyle -$}}}
  {\vcenter{\hbox{$\scriptscriptstyle -$}}}
  {\vcenter{\hbox{$\scriptscriptstyle -$}}}
}}
\newcommand{\Lp}{L^p}
\newcommand{\Lploc}{L^{p}\loc}
\newcommand{\Np}{N^{1,p}}
\newcommand{\Nploc}{N^{1,p}\loc}
\newcommand{\Dp}{D^p}
\newcommand{\Dploc}{D^{p}\loc}
\newcommand{\ut}{\tilde{u}}
\newcommand{\uHp}{\itoverline{P}}   
\newcommand{\lHp}{\itunderline{P}}  
\newcommand{\oHp}{H}                
\newcommand{\Hp}{P}                 

\RequirePackage{mathrsfs}
\newcommand{\K}{{\mathscr{K}}}
\newcommand{\LL}{\mathscr{L}}%
\newcommand{\U}{\mathscr{U}}%
\DeclareMathOperator*{\essliminf}{ess\,lim\,inf}
\DeclareMathOperator*{\esslimsup}{ess\,lim\,sup}
\DeclareMathOperator*{\essinf}{ess\,inf}
\DeclareMathOperator*{\esssup}{ess\,sup}
\newcommand{\cpesssup}{\text{$\Cp$-}\esssup}
\newcommand{\cplimsup}{\text{$\Cp$-}\esslimsup}
\newcommand{\px}{{\ensuremath{p(\cdot)}}}
\numberwithin{equation}{section}
\newcommand{\imp}{\ensuremath{\Rightarrow} }

\newenvironment{ack}{\medskip{\it Acknowledgement.}}{}

\begin{document}

\authortitle{Anders Bj\"orn and Daniel Hansevi}
{Boundary regularity for \p-harmonic functions 
on unbounded sets in metric spaces}
{Boundary regularity for \p-harmonic functions \\ 
and solutions of obstacle problems \\
on unbounded sets in metric spaces}

\author
{Anders Bj\"orn \\
\it\small Department of Mathematics, Link\"oping University, \\
\it\small SE-581 83 Link\"oping, Sweden\/{\rm ;}
\it \small anders.bjorn@liu.se
\\
\\
Daniel Hansevi \\
\it\small Department of Mathematics, Link\"oping University, \\
\it\small SE-581 83 Link\"oping, Sweden\/{\rm ;}
\it \small daniel.hansevi@liu.se
}

\date{Preliminary version, \today}
\date{}

\maketitle

\noindent{\small
{\bf Abstract}. 
The theory of boundary regularity for \p-harmonic functions is extended to 
unbounded open sets in complete metric spaces with a doubling measure 
supporting a \p-Poincar\'e inequality, 
$1<p<\infty$. 
The barrier classification of regular boundary points is established, 
and it is shown that 
regularity is a local property of the boundary. 
We also obtain boundary regularity results 
for solutions of the obstacle problem on open sets, 
and characterize regularity further 
in several other ways.}

\bigskip

\noindent {\small \emph{Key words and phrases}:
barrier, 
boundary regularity, 
Dirichlet problem, 
doubling measure, 
Kellogg property, 
metric space, 
nonlinear potential theory, 
obstacle problem, 
Perron solution, 
\p-harmonic function, 
Poincar\'e inequality.
}

\medskip


\section{Introduction}
\label{sec:intro} 
Let $\Omega\subset\R^n$ be a nonempty bounded open set 
and let $f\in C(\bdy\Omega)$. 
The Perron method 
(introduced on $\R^2$ in 1923 by Perron~\cite{Perron23} 
and independently by Remak~\cite{remak}) 
provides a unique function $\Hp f$ 
that is harmonic in $\Omega$ and takes the boundary values $f$ 
in a weak sense, i.e., 
$\Hp f$ is a solution of the Dirichlet problem for the Laplace equation. 
A point $x_0\in\bdy\Omega$ is 
\emph{regular} if 
$\lim_{\Omega\ni y\to x_0}\Hp f(y)=f(x_0)$ for all
$f\in C(\bdy\Omega)$. 
Regular boundary points were characterized in 1924 
by the so-called Wiener criterion and in terms of barriers, 
by Wiener~\cite{Wiener} and Lebesgue~\cite{Lebesgue}, respectively. 

A nonlinear analogue is to consider the 
Dirichlet problem for \p-harmonic functions, 
which are solutions of the \p-Laplace equation 
$\Delta_p u:=\Div(|\nabla u|^{p-2}\,\nabla u)=0$, $1<p<\infty$. 
This leads to a nonlinear potential theory, 
which has been studied since the 1960s, 
initially for $\R^n$, and later generalized to weighted $\R^n$, 
Riemannian manifolds, and other settings. 
For an extensive treatment in weighted $\R^n$, 
the reader may consult the monograph 
Heinonen--Kilpel\"ainen--Martio~\cite{HeKiMa}. 

More recently, 
nonlinear potential theory has been developed on 
complete metric spaces equipped with a doubling measure 
supporting a \p-Poincar\'e inequality, $1<p<\infty$, 
see, e.g., the monograph 
Bj\"orn--Bj\"orn~\cite{BBbook} and the references therein. 
The Perron method was extended to this setting by 
Bj\"orn--Bj\"orn--Shanmugalingam~\cite{BBS2} 
for bounded open sets and Hansevi~\cite{hansevi2} for 
unbounded open sets. 
Note that when $\R^n$ is equipped with a measure $d\mu=w\,dx$,
our assumptions on $\mu$ are equivalent to assuming
that $w$ is \p-admissible as in \cite{HeKiMa}, 
and our definition of \p-harmonic functions is equivalent to 
the one in \cite{HeKiMa},
see Appendix~A.2 in \cite{BBbook}. 

Boundary regularity for \p-harmonic functions on metric spaces 
was first studied by Bj\"orn~\cite{BjIll} and 
Bj\"orn--MacManus--Shan\-mu\-ga\-lin\-gam~\cite{BMS}. 
Bj\"orn--Bj\"orn--Shan\-mu\-ga\-lin\-gam~\cite{BBS} obtained the Kellogg property 
saying that the set 
of irregular boundary points has capacity zero. 
Bj\"orn--Bj\"orn~\cite{BB} 
obtained the barrier characterization, 
showed that regularity is a local property, 
and also studied boundary regularity for obstacle problems showing 
that they have essentially the same regular boundary points 
as the Dirichlet problem. 
These studies were pursued on bounded open sets.

In this paper we study boundary regularity for \p-harmonic functions 
on unbounded sets $\Omega$ in metric spaces $X$ 
(satisfying the assumptions above). 
The boundary $\bdy\Omega$ is considered within the one-point 
compactification $X^*=X\cup\{\infty\}$ of $X$, 
and is in particular always compact. 
We also impose the condition that the capacity $\Cp(X\setm\Omega)>0$. 

In this generality it is not known if continuous functions $f$ 
are resolutive, i.e., whether the upper and lower Perron solutions 
$\uHp_{\Omega}f$ and $\lHp_{\Omega} f$ 
coincide. 
We therefore make the following definition. 
\begin{definition}\label{def:reg}
A boundary point $x_0\in\bdy\Omega$ is \emph{regular} if 
\[
	\lim_{\Omega\ni y\to x_0}\uHp f(y)
	= f(x_0)
	\quad\text{for all }f\in C(\bdy\Omega).
\]
\end{definition}
With a few exceptions, we limit ourselves to studying 
regularity at finite boundary points. 

Our main results can be summarized as follows. 
\begin{theorem}
\label{thm-intro}
Let $x_0\in\bdyOmegaX$  
and let $B=B(x_0,r)$ for some $r>0$. 
\begin{enumerate}
\item\label{intro-kellogg}
The Kellogg property holds\textup{,} i.e., 
$\Cp(I\setm\{\infty\})=0$\textup{,} 
where $I$ is the set of irregular boundary points. 
\item\label{intro-barrier}
$x_0$ is regular if and only if there is a barrier at $x_0$. 
\item\label{intro-local}
Regularity is a local property\textup{,} i.e., 
$x_0$ is regular with respect to $\Omega$ if and only 
if it is regular with respect to $B\cap\Omega$. 
\end{enumerate}
\end{theorem}
Once the barrier characterization 
\ref{intro-barrier} 
has been shown, the locality \ref{intro-local} follows easily. 
Our proofs of these facts are however intertwined, 
and even though we use that these facts are already 
known to hold for bounded open sets, our proof 
is significantly longer than the proof 
in Bj\"orn--Bj\"orn~\cite{BB} (or \cite{BBbook}). 
On the other hand, once \ref{intro-local} has been deduced, 
\ref{intro-kellogg}
follows from its version for bounded domains. 
Several other characterizations 
of regularity are also given, 
see Sections~\ref{sec:bdy-regularity} and~\ref{sec:bdy-reg-further}. 

We also study the associated (one-sided) obstacle problem 
with prescribed boundary values $f$ and an obstacle $\psi$, 
where the solution is required 
to be greater than or equal to $\psi$ q.e.\ in $\Omega$
(i.e., up to a set of capacity zero). 
This problem obviously reduces to the Dirichlet problem for 
\p-harmonic functions when $\psi\equiv-\infty$. 
In Section~\ref{sec:bdy-reg-obst}, 
we show that if $x_0\in\bdyOmegaX$ 
is a regular boundary point 
and $f$ is continuous at $x_0$, 
then the solution $u$ of the obstacle problem 
attains the boundary value at 
$x_0$ in the limit, i.e., 
\[
	\lim_{\Omega\ni y\to x_0}u(y)
	= f(x_0)
\]
if and only if 
$\cplimsup_{\Omega\ni y\to x_0}\psi(y)\leq f(x_0)$. 
The results in  Section~\ref{sec:bdy-reg-obst} 
generalize the corresponding results in 
Bj\"orn--Bj\"orn~\cite{BB} to unbounded sets, 
with some improvements also for bounded sets. 
These results are new even on unweighted $\R^n$. 

Boundary regularity for \p-harmonic functions 
on $\R^n$ was first studied 
by Maz{\cprime}ya~\cite{Maz70} who obtained the 
sufficiency part of the Wiener criterion in 1970. 
Later on the full Wiener criterion has been obtained 
in various situations including weighted $\R^n$ and for Cheeger \p-harmonic 
functions on metric spaces, 
see \cite{KiMa94}, \cite{Lind-Mar}, \cite{Mikkonen},
and \cite{JB-Matsue}. 
The full Wiener criterion 
for \p-harmonic functions defined using upper gradients 
remains open even for bounded open sets in 
metric spaces (satisfying the assumptions above), 
but the sufficiency has been obtained, 
see \cite{BMS} and \cite{JB-pfine}, 
and a weaker necessity condition, see \cite{JBCalcVar}. 
An important consequence of 
Theorem~\ref{thm-intro}\,\ref{intro-local} 
is that the sufficiency part of the Wiener criterion 
holds for unbounded open sets. 
(Hence also the porosity-type conditions in 
Corollary~11.25 in \cite{BBbook} 
imply regularity for unbounded open sets.) 

In nonlinear potential theory, the Kellogg property was first obtained by 
Hedberg~\cite{Hedb} and Hedberg--Wolff~\cite{HedWol} on $\R^n$ 
(see also Kilpel\"ainen~\cite{Kilp89}). 
It was 
extended to homogeneous spaces by Vodop{\cprime}\-yanov~\cite{Vodopyanov89}, 
to weighted $\R^n$ by 
Hei\-no\-nen--Kil\-pe\-l\"ai\-nen--Martio~\cite{HeKiMa}, 
to subelliptic equations by 
Markina--Vodop{\cprime}yanov~\cite{MV2}, 
and to bounded open sets in metric spaces by 
Bj\"orn--Bj\"orn--Shan\-mu\-ga\-lin\-gam~\cite{BBS}. 
In some of these papers boundary regularity was defined 
in a different way than through Perron solutions, 
but these definitions are now known to be equivalent. 
See also 
\cite{ABB} and \cite{LLT} for the Kellogg property for 
$\px$-harmonic functions on $\R^n$.

Granlund--Lindqvist--Martio~\cite{GLM86} were the first 
to define boundary regularity using Perron solutions for \p-harmonic 
functions, $p\neq 2$. 
They studied the case $p=n$ in $\R^n$ and obtained 
the barrier characterization 
in this case for bounded open sets. 
Kilpel\"ainen~\cite{Kilp89} generalized the barrier characterization 
to $p>1$ and 
also deduced resolutivity for continuous functions. 
The results in \cite{Kilp89} 
covered both  bounded and unbounded open sets
in unweighted $\R^n$,
and were extended 
to weighted $\R^n$ (with a \p-admissible measure) in 
Hei\-no\-nen--Kil\-pe\-l\"ai\-nen--Martio~\cite[Chapter~9]{HeKiMa}.

As already mentioned, the Perron method for \p-harmonic functions 
was extended to 
metric spaces in Bj\"orn--Bj\"orn--Shan\-mu\-ga\-lin\-gam~\cite{BBS2} 
and Hansevi~\cite{hansevi2}. 
It has also been extended 
to other types of boundaries 
in \cite{BBSdir}, \cite{BBSjodin}, \cite{ES}, and \cite{ABcomb}. 
Various aspects of boundary regularity for \p-harmonic functions 
on bounded open sets in metric spaces have 
also been studied in 
\cite{ASh18}, \cite{ABclass}--\cite{BB2} and \cite{BBL2}. 

Very recently, Bj\"orn--Bj\"orn--Li~\cite{BBLi} studied Perron solutions and boundary regular for \p-harmonic functions on unbounded open sets in 
Ahlfors regular metric spaces. 
There is some overlap with the results in this paper, 
but it is not substantial and 
here we consider more general metric spaces than in \cite{BBLi}.

\begin{ack}
The first author was supported by the Swedish Research Council,
grant 2016-03424.
\end{ack}

\section{Notation and preliminaries}
\label{sec:prel} 
We assume that $(X,d,\mu)$ 
is a metric measure space (which we simply refer to as $X$) 
equipped with a metric $d$ and a 
positive complete Borel measure $\mu$ such that 
$0<\mu(B)<\infty$ 
for every ball $B\subset X$. 
It follows that $X$ is separable, second countable, and Lindel\"of 
(these properties are equivalent for metric spaces). 
For balls 
$B(x_0,r):=\{x\in X:d(x,x_0)<r\}$, 
we let $\lambda B=\lambda B(x_0,r):=B(x_0,\lambda r)$ 
for $\lambda>0$. 
The $\sigma$-algebra 
on which $\mu$ is defined 
is the completion of the Borel $\sigma$-algebra. 
We also assume that $1<p<\infty$. 
Later we will impose further requirements 
on the space and on the measure. 
We will keep the discussion short, 
see the monographs
Bj\"orn--Bj\"orn~\cite{BBbook} and
Heinonen--Koskela--Shanmugalingam--Tyson~\cite{HKSTbook} 
for proofs, 
further discussion, 
and references on the topics in this section. 

The measure $\mu$ is 
\emph{doubling} if there exists 
a constant $C\geq 1$ such that 
\[
	0 
	< \mu(2B) 
	\leq C\mu(B) 
	< \infty
\]
for every ball $B\subset X$. 
A metric space is 
\emph{proper} 
if all bounded closed subsets are compact, 
and this is in particular true if the metric space 
is complete and the measure is doubling. 

We use the standard notation 
$f_\limplus=\max\{f,0\}$ and 
$f_\limminus=\max\{-f,0\}$, and 
let $\chi_E$ denote the characteristic function 
of the set $E$. 
Semicontinuous functions are allowed 
to take values in 
$\eR:=[-\infty,\infty]$, 
whereas continuous functions will be assumed to be real-valued 
unless otherwise stated. 
For us, a curve in $X$ is a rectifiable nonconstant continuous mapping 
from a compact interval into $X$, 
and it can thus be parametrized 
by its arc length $\ds$. 

By saying that a property holds for \emph{\p-almost every curve}, 
we mean that it fails only for 
a curve family $\Gamma$ with zero \p-modulus, 
i.e., 
there exists a nonnegative $\rho\in\Lp(X)$ such that 
$\int_\gamma\rho\,\ds=\infty$ for every curve $\gamma\in\Gamma$. 

Following Koskela--MacManus~\cite{KoMac98} 
we make the following definition, 
see also Heinonen--Koskela~\cite{HeKo98}.
\begin{definition}\label{def:upper-gradients}
A measurable function $g\colon X\to[0,\infty]$ is a 
\emph{\p-weak upper gradient} 
of the function $f\colon X\to\eR$ 
if 
\[
	|f(\gamma(0)) - f(\gamma(l_{\gamma}))| 
	\leq \int_{\gamma}g\,\ds
\]
for \p-almost every curve 
$\gamma\colon[0,l_{\gamma}]\to X$, 
where we use the convention that the left-hand side is $\infty$ 
when at least one of the terms on the left-hand side is infinite. 
\end{definition}
Shanmugalingam~\cite{Shanmugalingam00} 
used \p-weak upper gradients to 
define so-called Newtonian spaces. 
\begin{definition}\label{def:Newtonian-space}
The \emph{Newtonian space} on $X$, 
denoted $\Np(X)$, 
is the space of all 
extended real-valued functions $f\in\Lp(X)$ 
such that 
\[
	\|f\|_{\Np(X)} 
	:= \biggl(\int_X|f|^p\,\dmu + \inf_g\int_X g^p\,\dmu\biggr)^{1/p}<\infty, 
\]
where the infimum is taken over all \p-weak upper gradients $g$ of $f$. 
\end{definition}
Shanmugalingam~\cite{Shanmugalingam00} proved that 
the associated quotient space $\Np(X)/\sim$ 
is a Banach space, where $f\sim h$ 
if and only if $\|f-h\|_{\Np(X)}=0$. 
In this paper we assume that functions in $\Np(X)$ 
are defined everywhere (with values in $\eR$), 
not just up to an equivalence class. 
This is important, in particular for the definition of 
\p-weak upper gradients to make sense. 
\begin{definition}\label{def:Dirichlet-space}
An everywhere defined, measurable, extended real-valued function on $X$ 
belongs to the 
\emph{Dirichlet space} $\Dp(X)$ 
if it has a \p-weak upper gradient in $\Lp(X)$. 
\end{definition}
A measurable set $A\subset X$ can be 
considered to be a metric space in its own right 
(with the restriction of $d$ and $\mu$ to $A$). 
Thus the Newtonian space $\Np(A)$ and the Dirichlet space $\Dp(A)$ 
are also given by 
Definitions~\ref{def:Newtonian-space}~and~\ref{def:Dirichlet-space}, 
respectively. 
If $X$ is proper and $\Omega\subset X$ is open, 
then $f\in\Nploc(\Omega)$ 
if and only if 
$f\in\Np(V)$ 
for every open $V$ such that $\clV$ 
is a compact subset of $\Omega$, 
and similarly for $\Dploc(\Omega)$. 
If $f\in\Dploc(X)$, 
then $f$ has a \emph{minimal \p-weak upper gradient} $g_f\in\Lploc(X)$ 
in the sense that $g_f\leq g$ a.e.\ for 
all \p-weak upper gradients $g\in\Lploc(X)$ of $f$. 
\begin{definition}\label{def:capacity}
The (\emph{Sobolev}) \emph{capacity} of a set $E\subset X$ is the number 
\[
	\Cp(E) 
	:= \inf_{f}\|f\|_{\Np(X)}^p,
\]
where the infimum is taken over all 
$f\in\Np(X)$ such that $f\geq 1$ on $E$. 

Whenever a property holds for all points 
except for those in a set of capacity zero, 
it is said to hold \emph{quasieverywhere} (\emph{q.e.}). 
\end{definition}
The capacity is countably subadditive, 
and it is the correct gauge 
for distinguishing between two Newtonian functions: 
If $f\in\Np(X)$, then $f\sim h$ if and only if $f=h$ q.e. 
Moreover, 
if $f,h\in\Nploc(X)$ and $f=h$ a.e., then $f=h$ q.e. 

There is a subtle, but important, difference to the standard 
theory on $\R^n$ where the equivalence classes in the Sobolev space 
are (usually) up to sets of measure zero, while here the equivalence 
classes in $\Np(X)$ are up to sets of capacity zero. 
Moreover, under the 
assumptions from the beginning of Section~\ref{sec:p-harmonic}, 
the functions in $\Nploc(X)$ and $\Nploc(\Omega)$ are quasicontinuous. 
On weighted $\R^n$, the Newtonian space $\Np(X)$ therefore 
corresponds to the refined Sobolev space mentioned on p.~96 in 
Hei\-no\-nen--Kil\-pe\-l\"ai\-nen--Martio~\cite{HeKiMa}. 

In order to be able to compare boundary values 
of Dirichlet and Newtonian functions, 
we need the following spaces. 
\begin{definition}\label{def:Dp0}
For subsets $E$ and $A$ of $X$, 
where $A$ is measurable, 
the \emph{Dirichlet space with zero boundary values in $A\setm E$}, 
is 
\[
	\Dp_0(E;A) 
	:= \{f	|_{E\cap A}:f	\in\Dp(A)\textup{ and }f	=0\textup{ in }A\setm E\}.
\]
The \emph{Newtonian space with zero boundary values} $\Np_0(E;A)$ 
is defined analogously. 
We let $\Dp_0(E)$ and $\Np_0(E)$ denote 
$\Dp_0(E;X)$ and $\Np_0(E;X)$, respectively. 
\end{definition}
The condition ``$f=0$ in $A\setm E$'' 
can in fact be replaced by ``$f=0$ q.e.\ in $A\setm E$'' 
without changing the obtained spaces. 
\begin{definition}\label{def:Poincare-inequality}
We say that $X$ supports a \p-\emph{Poincar\'e inequality} 
if there exist constants, 
$C>0$ and $\lambda\geq 1$ (the dilation constant), 
such that 
\begin{equation}\label{def:Poincare-inequality-ineq}
	\vint_B|f-f_B|\,\dmu
	\leq C\diam(B)\biggl(\vint_{\lambda B}g^p\,\dmu\biggr)^{1/p}
\end{equation}
for all balls $B\subset X$, 
all integrable functions $f$ on $X$, 
and all \p-weak upper gradients $g$ of $f$. 
\end{definition}
In \eqref{def:Poincare-inequality-ineq}, 
we have used the convenient notation 
$f_B:=\vint_B f\,\dmu:=\frac{1}{\mu(B)}\int_B f\,\dmu$. 
Requiring a Poincar\'e inequality to hold is one way of 
making it possible to control functions by their 
\p-weak upper gradients.

\section{The obstacle problem and \texorpdfstring{\boldmath$p\mspace{1mu}$}{p}-harmonic functions}
\label{sec:p-harmonic}
\emph{We assume from now on 
that\/ $1<p<\infty$\textup{,} 
that $X$ is a complete metric measure space 
supporting a \p-Poincar\'e inequality\textup{,} 
that $\mu$ is doubling\textup{,} 
and that\/ $\Omega\subset X$ 
is a nonempty \textup{(}possibly unbounded\textup{)} 
open subset with $\Cp(X\setm\Omega)>0$.}

\medskip

One of our fundamental tools is 
the following obstacle problem, 
which in this generality was first considered by Hansevi~\cite{hansevi1}.
\begin{definition}\label{def:obst}
Let $V\subset X$ be a nonempty open subset with $\Cp(X\setm V)>0$. 
For $\psi\colon V\to\eR$ and $f\in\Dp(V)$, 
let 
\[
	\K_{\psi,f}(V)
	= \{v\in\Dp(V):v-f\in\Dp_0(V)\textup{ and }v\geq\psi\text{ q.e.\ in }V\}. 
\]
We say that $u\in\K_{\psi,f}(V)$ is a 
\emph{solution of the }$\K_{\psi,f}(V)$-\emph{obstacle problem 
\textup{(}with obstacle $\psi$ and boundary values $f$\,\textup{)}} 
if 
\[
	\int_V g_u^p\,\dmu 
	\leq \int_V g_v^p\,\dmu
	\quad\textup{for all }v\in\K_{\psi,f}(V).
\]
When $V=\Omega$, 
we usually denote $\K_{\psi,f}(\Omega)$ by $\K_{\psi,f}$.
\end{definition}
It was proved in Hansevi~\cite[Theorem~3.4]{hansevi1} that 
the $\K_{\psi,f}$-obstacle problem has a unique 
(up to sets of capacity zero) solution 
whenever $\K_{\psi,f}$ is nonempty. 
Furthermore, 
in this case, 
there is a unique lsc-regularized solution of the 
$\K_{\psi,f}$-obstacle problem, 
by Theorem~4.1 in \cite{hansevi1}. 
A function $u$ is \emph{lsc-regularized} if $u=u^*$, 
where the \emph{lsc-regularization} $u^*$ of $u$ is 
defined by 
\[
	u^*(x) 
	= \essliminf_{y\to x}u(y) 
	:= \lim_{r\to 0}\essinf_{B(x,r)}u.
\]
\begin{definition}\label{def:min}
A function $u\in\Nploc(\Omega)$ 
is a \emph{minimizer} in $\Omega$ if 
\begin{equation} \label{eq-minimizer}
	\int_{\phi\neq 0}g_u^p\,\dmu 
	\leq \int_{\phi\neq 0}g_{u+\phi}^p\,\dmu
	\quad\text{for all }\phi\in\Np_0(\Omega). 
\end{equation}
If \eqref{eq-minimizer} only holds for all 
nonnegative $\phi\in\Np_0(\Omega)$, 
then $u$ is a \emph{superminimizer}.

Moreover, 
a function is \emph{\p-harmonic} 
if it is a continuous minimizer. 
\end{definition}
Kinnunen--Shanmugalingam~\cite[Proposition~3.3 and Theorem~5.2]{KiSh01} 
used De Giorgi's method to 
show that every minimizer 
$u$ has a H\"older continuous representative 
$\ut$ such that $\ut=u$ q.e. 
They also obtained the strong maximum principle \cite[Corollary~6.4]{KiSh01} 
for \p-harmonic functions. 
Bj\"orn--Marola~\cite[p.\ 362]{BMarola} obtained the same conclusions 
using Moser iterations. 
See alternatively Theorems~8.13 and 8.14 in \cite{BBbook}. 
Note that $\Nploc(\Omega)=\Dploc(\Omega)$ 
(under our assumptions), 
by Proposition~4.14 in \cite{BBbook}.

If $\psi\colon\Omega\to[-\infty,\infty)$ 
is continuous as an extended real-valued function, 
and $\K_{\psi,f}\neq\emptyset$, 
then the lsc-regularized solution of the $\K_{\psi,f}$-obstacle problem 
is continuous, 
by Theorem~4.4 in Hansevi~\cite{hansevi1}. 
Thus the following definition makes sense. 
It was first used in this generality by 
Hansevi~\cite[Definition~4.6]{hansevi1}.
\begin{definition}\label{def:ext}
Let $V\subset X$ be a nonempty open set with $\Cp(X\setm V)>0$. 
The \emph{\p-harmonic extension} 
$\oHp_V f$ of $f\in\Dp(V)$ to $V$ is the continuous solution 
of the $\K_{-\infty,f}(V)$-obstacle problem. 
When $V=\Omega$, we usually write $\oHp f$ instead of $\oHp_\Omega f$.
\end{definition}
\begin{definition}\label{def:superharm}
A function $u\colon\Omega\to(-\infty,\infty]$ 
is \emph{superharmonic} in $\Omega$ if 
\begin{enumerate}
\renewcommand{\theenumi}{\textup{(\roman{enumi})}}%
\item $u$ is lower semicontinuous; 
\item $u$ is not identically $\infty$ in any component of $\Omega$; 
\item for every nonempty open set 
$V$ such that $\clV$ is a compact subset of $\Omega$ 
and all $v\in\Lip(\clV)$, 
we have $\oHp_{V}v\leq u$ in $V$ 
whenever $v\leq u$ on $\bdy V$. 
\end{enumerate}
A function $u\colon\Omega\to[-\infty,\infty)$ is 
\emph{subharmonic} if $-u$ is superharmonic.
\end{definition}
There are several other equivalent definitions of superharmonic functions, 
see, e.g., 
Theorem~6.1 in Bj\"orn~\cite{ABsuper} 
(or Theorem~9.24 and Propositions~9.25 and~9.26 in \cite{BBbook}). 

An lsc-regularized 
solution of the obstacle problem is always superharmonic, 
by Proposition~3.9 in Hansevi~\cite{hansevi1} 
together with Proposition~7.4 in Kinnunen--Martio~\cite{KiMa02} 
(or Proposition~9.4 in \cite{BBbook}). 
On the other hand, 
superharmonic functions are always lsc-regularized, 
by Theorem~7.14 in Kinnunen--Martio~\cite{KiMa02} 
(or Theorem~9.12 in \cite{BBbook}). 

When proving Theorem~\ref{thm:reg-3} we will need the following generalization 
of Proposition~7.15 in \cite{BBbook}, 
which may be of independent interest. 
\begin{lemma}\label{lem-super-obst-V}
Let $u$ be superharmonic in $\Omega$ 
and let $V\subset\Omega$ be a bounded nonempty open subset such that 
$\Cp(X\setm V)>0$ and $u\in\Dp(V)$. 
Then $u$ is the lsc-regularized solution of the 
$\K_{u,u}(V)$-obstacle problem. 
\end{lemma}
The boundedness assumption cannot be dropped. 
To see this, let $1<p<n$ and 
$\Omega=V=\R^n\setm\itoverline{B(0,1)}$ 
in unweighted $\R^n$. 
Then $u(x)= |x|^{(p-n)/(p-1)}$ 
is superharmonic in $\Omega$ and belongs to $\Dp(V)$. 
However, $v\equiv 1$ is the 
lsc-regularized solution of the $\K_{u,u}(V)$-obstacle problem. 
\begin{proof}
Corollary~9.10 in \cite{BBbook} implies that 
$u$ is superharmonic in $V$, 
and hence it follows from 
Corollary~7.9 and Theorem~7.14 in Kinnunen--Martio~\cite{KiMa02} 
(or Corollary~9.6 and Theorem~9.12 in \cite{BBbook}) 
that 
$u$ is an lsc-regularized superminimizer in $V$. 
Because $u\in\Dp(V)$, 
it is clear that 
$u\in\K_{u,u}(V)$. 
Let $v\in\K_{u,u}(V)$ and let $w=\max\{u,v\}$. 
Then $\phi:=w-u=(v-u)_\limplus\in\Dp_0(V)$, 
and since $X$ supports a \p-Friedrichs inequality 
(Definition~2.6 in Bj\"orn--Bj\"orn~\cite{BBnonopen}) and $V$ is bounded, 
we have $\phi\in\Np_0(V)$, 
by Proposition~2.7 in \cite{BBnonopen}. 
Because $v=w$ q.e.\ in $V$, 
it follows from Definition~\ref{def:min} that 
\[
	\int_V g_u^p\,\dmu
	\leq \int_V g_{u+\phi}^p\,\dmu
	= \int_V g_w^p\,\dmu
	= \int_V g_v^p\,\dmu.
\]
Hence $u$ is the lsc-regularized solution of the 
$\K_{u,u}(V)$-obstacle problem. 
\end{proof}

\section{Perron solutions}
\label{sec:perron} 
\emph{In addition to the assumptions given at the 
beginning of Section~\ref{sec:p-harmonic}\textup{,} 
from now on we make the convention that if\/ $\Omega$ is unbounded\textup{,} 
then the point at infinity\textup{,} $\infty$\textup{,} 
belongs to the boundary $\bdy\Omega$. 
Topological notions should therefore be understood with respect to the 
one-point compactification $X^*:=X\cup\{\infty\}$.} 

\medskip

Since continuous functions are assumed to be real-valued, 
every function in $C(\bdy\Omega)$ is bounded 
even if $\Omega$ is unbounded.
\begin{definition}\label{def:Perron}
Given a function $f\colon\bdy\Omega\to\eR$, 
let $\U_f(\Omega)$ be the collection of all 
functions 
$u$ that are superharmonic in $\Omega$, 
bounded from below, and such that 
\[
	\liminf_{\Omega\ni y\to x}u(y)
	\geq f(x)
	\quad\textup{for all }x\in\bdy\Omega.
\]
The \emph{upper Perron solution} of $f$ is defined by 
\[
	\uHp_\Omega f(x) 
	= \inf_{u\in\U_f(\Omega)}u(x), 
	\quad x\in\Omega.
\]
Let $\LL_f(\Omega)$ be the collection of all 
functions 
$v$ that are subharmonic in $\Omega$, 
bounded from above, and such that 
\[
	\limsup_{\Omega\ni y\to x}v(y) 
	\leq f(x) 
	\quad\textup{for all }x\in\bdy\Omega. 
\]
The \emph{lower Perron solution} of $f$ is defined by 
\[
	\lHp_\Omega f(x) 
	= \sup_{v\in\LL_f(\Omega)}v(x), 
	\quad x\in\Omega.
\]
If $\uHp_\Omega f=\lHp_\Omega f$, 
then we denote the common value by $\Hp_\Omega f$. 
Moreover, if $\Hp_\Omega f$ is real-valued, then $f$ is said to be 
\emph{resolutive} (with respect to $\Omega$). 
We often write $\Hp f$ instead of $\Hp_\Omega f$, 
and similarly for $\uHp f$ and $\lHp f$.
\end{definition}
Immediate consequences of the definition are: 
$\lHp f=-\uHp(-f)$ and $\uHp f\leq\uHp h$ 
whenever $f\leq h$ on $\bdy\Omega$. 
If $\alpha\in\R$ and $\beta\geq 0$, 
then $\uHp(\alpha + \beta f)=\alpha+\beta\uHp f$. 
Corollary~6.3 in Hansevi~\cite{hansevi2} shows that $\lHp f\leq\uHp f$. 
In each component of $\Omega$, 
$\uHp f$ is either \p-harmonic or identically $\pm\infty$, 
by Theorem~4.1 in Bj\"orn--Bj\"orn--Shanmugalingam~\cite{BBS2} 
(or Theorem~10.10 in \cite{BBbook}); 
the proof is local and applies 
also to unbounded $\Omega$. 
\begin{definition}\label{def:p-para}
Assume that $\Omega$ is unbounded. 
Then $\Omega$ is \emph{\p-parabolic} if 
for every compact $K\subset\Omega$, 
there exist functions $u_j\in\Np(\Omega)$ such that 
$u_j\geq 1$ on $K$ for all $j=1,2,\ldots$\,, and 
\[
	\int_\Omega g_{u_j}^p\,\dmu 
	\to 0 
	\quad\text{as }j\to\infty. 
\]
Otherwise, 
$\Omega$ is \emph{\p-hyperbolic}. 
\end{definition}
For examples of \p-parabolic sets, see, e.g., 
Hansevi~\cite{hansevi2}. 
The main reason for introducing \p-parabolic sets 
in \cite{hansevi2} 
was to be able to obtain resolutivity results. 
We formulate this in a special case, which will be sufficient for us. 
\begin{theorem}\label{thm-hansevi2-main-res}
\textup{(\cite[Theorem~6.1]{BBS2} 
and~\cite[Theorems~7.5 and~7.8]{hansevi2})} 
Assume that\/ $\Omega$ is bounded or \p-parabolic.

If $f\in C(\bdy\Omega)$\textup{,} then $f$ is resolutive. 

If $f\in\Dp(X)$ and 
$f(\infty)$ is defined\/ \textup{(}with a value in $\eR$\textup{)}\textup{,} 
then $f$ is resolutive and $\Hp f=\oHp f$. 
\end{theorem}

Recall from Section~\ref{sec:prel} 
that under our standing assumptions, 
the equivalence classes in $\Dp(X)$ only contain 
quasicontinuous representatives. 
This fact is crucial for the 
validity of the second part of Theorem~\ref{thm-hansevi2-main-res}.

\section{Boundary regularity}
\label{sec:bdy-regularity}
For unbounded \p-hyperbolic sets resolutivity of 
continuous functions is not known, 
which will be an obstacle to overcome in some of our proofs below. 
This explains why regularity
was defined using upper Perron solutions
in Definition~\ref{def:reg}. 
In our definition 
it is not required  that $\Omega$ is bounded,
but if it is, 
then it follows from Theorem~\ref{thm-hansevi2-main-res} that 
it coincides with the definitions of regularity in 
Bj\"orn--Bj\"orn--Shanmugalingam~\cite{BBS}, \cite{BBS2}, 
and 
Bj\"orn--Bj\"orn~\cite{BB}, \cite{BBbook}, 
where regularity is defined using $\Hp f$ or $\oHp f$. 
Thus we can use the boundary regularity results 
from these papers 
when considering bounded sets. 

Since $\uHp f=-\lHp(-f)$, 
the same concept of regularity is obtained if we replace the 
upper Perron solution by the lower Perron solution in 
Definition~\ref{def:reg}.
\begin{theorem}\label{thm:reg}
Let $x_0\in\bdy\Omega$. 
Fix $x_1 \in X$ and define
$d_{x_0}\colon X^*\to[0,1]$ 
by 
\begin{equation}\label{eq-dx0}
	d_{x_0}(x) 
	= \begin{cases}
		\min\{d(x,x_0),1\}
			& \text{when }x\neq\infty, \\
		1 
			& \text{when }x=\infty,
	\end{cases}
\quad \text{if } x_0 \ne \infty,
\end{equation}
and 
\[
	d_\infty(x) 
	= \begin{cases}
		\exp(-d(x,x_1))
			& \text{when }x\neq\infty, \\
		0 
			& \text{when }x=\infty.
	\end{cases}
\]

Then the following are equivalent\/\textup{:} 
\begin{enumerate}
\item\label{reg-reg} 
The point $x_0$ is regular. 
\item\label{reg-dx0}
It is true that 
\[
	\lim_{\Omega\ni y\to x_0}\uHp d_{x_0}(y)
	= 0.
\]
\item\label{reg-semicont-x0}
It is true that 
\[
	\limsup_{\Omega\ni y\to x_0}\uHp f(y)
	\leq f(x_0)
\]
for all $f\colon\bdy\Omega\to[-\infty,\infty)$ 
that are bounded from above on 
$\bdy\Omega$ and upper semicontinuous at $x_0$. 
\item\label{reg-cont-x0}
It is true that 
\[
	\lim_{\Omega\ni y\to x_0}\uHp f(y)
	= f(x_0)
\]
for all $f\colon\bdy\Omega\to\R$ that are 
bounded on $\bdy\Omega$ and continuous at $x_0$. 
\item\label{reg-cont}
It is true that 
\[
	\limsup_{\Omega\ni y\to x_0}\uHp f(y)
	\leq f(x_0)
\]
for all $f\in C(\bdy\Omega)$. 
\end{enumerate}
\end{theorem}
The particular form of $d_{x_0}$ is not important. 
The same characterization 
holds for any nonnegative continuous function 
$d\colon X^*\to[0,\infty)$ which is zero at and only at $x_0$. 
For the later applications in this paper it will also 
be important that $d\in\Dp(X)$, 
which is true for $d_{x_0}$. 
\begin{proof}
\ref{reg-reg} $\imp$ \ref{reg-dx0}
This is trivial.

\smallskip

\ref{reg-dx0} $\imp$ \ref{reg-semicont-x0} 
Fix $\alpha>f(x_0)$. 
Since $f$ is upper semicontinuous at $x_0$, 
there exists an open set $U\subset X^*$ such that $x_0\in U$ and 
$f(x)<\alpha$ for all $x\in U\cap\bdy\Omega$. 
Let $\beta=\sup_{\bdy\Omega}(f-\alpha)_\limplus$ 
and $\delta:=\inf_{\bdy\Omega\setminus U}d_{x_0}>0$. 
(Note that $\delta=\infty$ if $\bdy\Omega\setm U=\emptyset$.) 
Then $\beta<\infty$ and 
$f\leq\alpha+\beta d_{x_0}/\delta$ on $\bdy\Omega$, 
and hence it follows that 
\[
	\limsup_{\Omega\ni y\to x_0}\uHp f(y)
	\leq \alpha+\frac{\beta}{\delta}\lim_{\Omega\ni y\to x_0}\uHp d_{x_0}(y)
	= \alpha.
\]
Letting $\alpha\to f(x_0)$ yields the desired result.

\smallskip

\ref{reg-semicont-x0} $\imp$ \ref{reg-cont-x0} 
Let $f$ be bounded on $\bdy\Omega$ and continuous at $x_0$. 
Both $f$ and $-f$ satisfy the hypothesis in \ref{reg-semicont-x0}. 
The conclusion in \ref{reg-cont-x0} follows as  
\[
	\limsup_{\Omega\ni y\to x_0}\uHp f(y)
	\leq f(x_0)
	\leq -\limsup_{\Omega\ni y\to x_0}\uHp(-f)(y)
	= \liminf_{\Omega\ni y\to x_0}\lHp f(y)
	\leq \liminf_{\Omega\ni y\to x_0}\uHp f(y).
\]

\smallskip

\ref{reg-cont-x0} $\imp$ \ref{reg-cont}
This is trivial.

\smallskip

\ref{reg-cont} $\imp$ \ref{reg-reg} 
This is analogous to the proof of 
\ref{reg-semicont-x0} $\imp$ \ref{reg-cont-x0}. 
\qedhere
\end{proof}
\emph{We will mainly concentrate on the regularity of finite points 
in the rest of the paper.}

\section{Barrier characterization of regular points}
\label{sec:barrier}
\begin{definition}\label{def:barrier}
A function $u$ is a \emph{barrier} (with respect to $\Omega$) 
at $x_0\in\bdy\Omega$ if 
\begin{enumerate}
\renewcommand{\theenumi}{\textup{(\roman{enumi})}}%
\item\label{barrier-i}
$u$ is superharmonic in $\Omega$; 
\item\label{barrier-ii}
$\lim_{\Omega\ni y\to x_0}u(y)=0$; 
\item\label{barrier-iii}
$\liminf_{\Omega\ni y\to x}u(y)>0$ for every $x\in\bdy\Omega\setm\{x_0\}$. 
\end{enumerate}
\end{definition}
Superharmonic functions satisfy the 
strong minimum principle, i.e., 
if $u$ is superharmonic and attains its minimum in some component 
$G$ of $\Omega$, then $u|_G$ is constant 
(see Theorem~9.13 in \cite{BBbook}). 
This implies that a barrier is always nonnegative, 
and furthermore, that a barrier is positive if 
every component $G\subset\Omega$ has a boundary point in $\bdy G\setm\{x_0\}$. 
\begin{theorem}\label{thm:barrier}
If $x_0\in\bdyOmegaX$ and $B$ is a ball such that $x_0\in B$\textup{,} 
then the following are equivalent\/\textup{:}
\begin{enumerate}
\item\label{barrier-reg-Om}
The point $x_0$ is regular.
\item\label{barrier-bar-Om}
There is a barrier at $x_0$.
\item\label{barrier-bar-cont-Om}
There is a positive continuous barrier at $x_0$.
\item\label{barrier-reg-B}
The point $x_0$ is regular with respect to $\Omega\cap B$.
\item\label{barrier-bar-B}
There is a positive 
barrier with respect to $\Omega\cap B$ at $x_0$.
\item\label{barrier-bar-cont-B}
There is a 
continuous barrier $u$ with respect to $\Omega\cap B$ at $x_0$\textup{,} 
such that $u(x)\geq d(x,x_0)$ for all $x\in\Omega\cap B$.
\end{enumerate}
\end{theorem}
We first show that parts 
\ref{barrier-bar-cont-Om} to \ref{barrier-bar-cont-B} 
are equivalent, 
and that \ref{barrier-bar-cont-Om} $\imp$ \ref{barrier-bar-Om}
$\imp$ \ref{barrier-reg-Om}.
To conclude the proof we then show that 
\ref{barrier-reg-Om} $\imp$ \ref{barrier-bar-cont-Om},
which is by far the most complicated part of the proof.

In the next section, 
we will use this characterization 
to obtain the Kellogg property for 
unbounded sets. 
In the proof below we will however need the Kellogg property 
for bounded sets, 
which for
metric spaces 
is due to 
Bj\"orn--Bj\"orn--Shanmugalingam~\cite[Theorem~3.9]{BBS}. 
(See alternatively \cite[Theorem~10.5]{BBbook}.) 

We do not know if the corresponding characterizations
of regularity at $\infty$ holds, 
but 
the proof below 
shows that the existence of
a barrier implies regularity also at $\infty$.
\begin{proof} 
\ref{barrier-bar-cont-Om} $\imp$ \ref{barrier-bar-B}
Suppose that $u$ is a positive barrier with respect to $\Omega$ at $x_0$. 
Then $u$ is superharmonic in $\Omega\cap B$, 
by Corollary~9.10 in \cite{BBbook}. 
Clearly, $u$ satisfies condition \ref{barrier-ii} in 
Definition~\ref{def:barrier} 
with respect to $\Omega\cap B$, 
and since $u$ is positive and lower semicontinuous in $\Omega$, 
$u$ also satisfies condition \ref{barrier-iii} 
in Definition~\ref{def:barrier} with respect to $\Omega\cap B$. 
Thus $u$ is a positive 
barrier with respect to $\Omega\cap B$ at $x_0$. 

\smallskip

\ref{barrier-bar-B} $\imp$ \ref{barrier-reg-B} 
This follows from 
Theorem~4.2 in Bj\"orn--Bj\"orn~\cite{BB} 
(or Theorem~11.2 
in \cite{BBbook}). 
Alternatively one can appeal to the proof of 
\ref{barrier-bar-Om} $\imp$ \ref{barrier-reg-Om} 
below.

\smallskip

\ref{barrier-reg-B} $\imp$ \ref{barrier-bar-cont-B} 
This follows from 
Theorem~6.1 in \cite{BB} 
(or Theorem~11.11 
in \cite{BBbook}). 

\smallskip

\ref{barrier-bar-cont-B} $\imp$ \ref{barrier-bar-cont-Om} 
Suppose that $u$ is a continuous barrier with respect to 
$\Omega\cap B$ at $x_0$ such that 
$u(x)\geq d(x,x_0)$ for all $x\in\Omega\cap B$. 
Let $m=\dist(x_0,X\setm B)$ and let 
\[
	v
	= \begin{cases}
		m & \text{in }\Omega\setm B, \\
		\min\{u,m\} & \text{in }\Omega\cap B.
	\end{cases}
\]
Then $v$ is continuous, 
and hence superharmonic in $\Omega$ 
by Lemma~9.3 in \cite{BBbook} 
and the pasting lemma for superharmonic functions, 
Lemma~3.13 in Bj\"orn--Bj\"orn--M\"ak\"a\-l\"ainen--Parviainen~\cite{BBMP} 
(or Lemma~10.27 in \cite{BBbook}). 
Furthermore, 
$v$ clearly satisfies conditions 
\ref{barrier-ii} and~\ref{barrier-iii} 
in Definition~\ref{def:barrier}, 
and is thus a positive continuous barrier 
with respect to $\Omega$ at $x_0$. 

\smallskip

\ref{barrier-bar-cont-Om} $\imp$ \ref{barrier-bar-Om} 
This implication is trivial. 

\smallskip

\ref{barrier-bar-Om} $\imp$ \ref{barrier-reg-Om}
Suppose that $x_0\in\bdy\Omega$.
(Thus we include the case $x_0=\infty$ when
proving this implication.) 
Let $f\in C(\bdy\Omega)$ and fix $\alpha>f(x_0)$. 
Then the set 
$U:=\{x\in\bdy\Omega:f(x)<\alpha\}$ is open relative to
 $\bdy\Omega$, 
and $\beta:=\sup_{\bdy\Omega}(f-\alpha)_\limplus<\infty$. 
Assume that $u$ is a  
barrier at $x_0$,  
and extend $u$ lower semicontinuously to the boundary by letting
\[
	u(x)
	= \liminf_{\Omega\ni y\to x}u(y),
	\quad x\in\bdy\Omega.
\]
Because $u$ is lower semicontinuous and satisfies condition 
\ref{barrier-iii} in Definition~\ref{def:barrier}, 
we have $\delta:=\inf_{\bdy\Omega\setm U}u>0$. 
(Note that $\delta=\infty$ if $\bdy\Omega\setm U=\emptyset$.) 
It follows that 
\[
	f 
	\leq \alpha+\frac{\beta u}{\delta}
	=: h
	\quad\text{on }\bdy\Omega.
\]
Since $h$ is bounded from below and superharmonic, 
we have $h\in\U_{f}$, 
and hence 
$\uHp f\leq h$ in $\Omega$. 
As $u$ is a barrier, 
it follows that 
\[
	\limsup_{\Omega\ni y\to x_0}\uHp f(y)
	\leq \alpha+\frac{\beta}{\delta}\lim_{\Omega\ni y\to x_0}u(y)
	= \alpha.
\]
Letting $\alpha\to f(x_0)$, and appealing to 
Theorem~\ref{thm:reg} 
shows that $x_0$ is regular. 

\smallskip 

\ref{barrier-reg-Om} $\imp$ \ref{barrier-bar-cont-Om}
Assume that $x_0$ is regular. 
We begin with the case when $\Cp(\{x_0\})>0$. 
Let 
$d_{x_0}\in\Dp(X)$ be given by \eqref{eq-dx0}. 
We let $u$ be the continuous solution of the 
$\K_{d_{x_0},d_{x_0}}$-obstacle problem, 
which is superharmonic (see Section~\ref{sec:p-harmonic}) 
and hence satisfies condition~\ref{barrier-i} in 
Definition~\ref{def:barrier}. 
We also extend $u$ to $X$ by letting $u=d_{x_0}$ outside $\Omega$ 
so that $u\in\Dp(X)$. 
Then $0\leq u\leq 1$ (as $0\leq d_{x_0}\leq 1$), 
and thus $U:=\{x\in\Omega:u(x)>d_{x_0}(x)\}\subset B(x_0,1)$. 
Since $u$ and $d_{x_0}$ are continuous, 
we see that $U$ is open and $u=d_{x_0}$ on $\bdy U$. 

Suppose that $x_0\in\bdy U$. 
Proposition~3.7 in Hansevi~\cite{hansevi1} implies that 
$u$ is the continuous solution of the 
$\K_{d_{x_0},d_{x_0}}(U)$-obstacle problem. 
Since $u>d_{x_0}$ in $U$, 
we have 
$u|_U=\oHp_U d_{x_0}$, 
and hence, by 
Theorem~\ref{thm-hansevi2-main-res}, 
\begin{equation}\label{barrier-HP-eq}
	u|_U
	= \oHp_U d_{x_0}
	= \uHp_U d_{x_0}.
\end{equation}
The Kellogg property for bounded sets 
(Theorem~3.9 in Bj\"orn--Bj\"orn--Shan\-mu\-ga\-lin\-gam~\cite{BBS} 
or Theorem~10.5 in \cite{BBbook}) 
implies that $x_0$ is regular with respect to $U$ 
as $\Cp(\{x_0\})>0$. 
It thus follows that 
\[
	\lim_{U\ni y\to x_0}u(y)
	= \lim_{U\ni y\to x_0}\uHp_U d_{x_0}(y)
	= 0.
\]

On the other hand, if $x_0\in\bdy(\Omega\setm U)$, 
then 
\[
	\lim_{\Omega\setm U\ni y\to x_0}u(y)
	= \lim_{\Omega\setm U\ni y\to x_0}d_{x_0}(y)
	= 0,
\]
and hence 
$u(y)\to 0$ as $\Omega\ni y\to x_0$ 
regardless of the position of $x_0$ on $\bdy\Omega$. 
(Note that it is possible that $x_0$ belongs to both 
$\bdy U$ and $\bdy(\Omega\setm U)$.) 
Thus $u$ satisfies condition~\ref{barrier-ii} in 
Definition~\ref{def:barrier}.

Furthermore, $u$ also satisfies condition~\ref{barrier-iii} in 
Definition~\ref{def:barrier}, as 
\[
	\liminf_{\Omega\ni y\to x}u(y)
	\geq \liminf_{\Omega\ni y\to x}d_{x_0}(y)
	= d_{x_0}(x)
	> 0
	\quad\text{for all }x\in\bdy\Omega\setm\{x_0\}.
\]
Thus $u$ is a positive continuous barrier at $x_0$.

\smallskip

Now we consider the case when $\Cp(\{x_0\})=0$. 
As the capacity $\Cp$ is an outer capacity, 
by Corollary~1.3 in Bj\"orn--Bj\"orn--Shanmugalingam~\cite{BBS5} 
(or \cite[Theorem~5.31]{BBbook}), 
$\lim_{r\to 0}\Cp(B(x_0,r))=0$. 
This, together with the fact that 
$\Cp(X\setm\Omega)>0$, 
shows that 
we can find a ball $B:=B(x_0,r)$ 
with sufficiently small radius $r>0$ so that 
$\Cp(X\setm(\Omega\cup 2B))>0$. 
Let $h\colon X\to[-r,0]$ be defined by 
\[
	h(x) 
	= -\min\{d(x,x_0),r\}.
\]
Let $v$ be the continuous solution of the 
$\K_{h,h}(\Omega\cup 2B)$-obstacle problem, 
and extend $v$ to $X$ by letting $v=h$ outside $\Omega\cup 2B$. 
Then $-r\leq h\leq v\leq v(x_0)=0$ in $\Omega\cup 2B$. 
Let $u=\lHp_\Omega w$, 
where 
\begin{equation}\label{barrier-def-w}
	w(x)
	:= \begin{cases}
			-v(x), & x\in\Omega, \\
			-\liminf\limits_{\Omega\ni y\to x}v(y), & x\in\bdy\Omega.
		\end{cases}
\end{equation}
Then $u$ is \p-harmonic, see Section~\ref{sec:perron}, 
and in particular continuous. 
Thus $u$ satisfies condition~\ref{barrier-i} in 
Definition~\ref{def:barrier}. 

Because $x_0$ is regular and $w$ is continuous at $x_0$ and bounded, 
it follows from 
Theorem~\ref{thm:reg} that 
$u$ satisfies condition~\ref{barrier-ii} in 
Definition~\ref{def:barrier}, as 
\[
	\lim_{\Omega\ni y\to x_0}u(y)
	= \lim_{\Omega\ni y\to x_0}\lHp_\Omega w(y)
	= -\lim_{\Omega\ni y\to x_0}\uHp_\Omega(-w)(y)
	= w(x_0)
	= 0.
\]

Let $V=\{x\in\Omega\cup 2B:v(x)>h(x)\}$. 
Clearly, $v=h<0$ in $((\Omega\cup 2B)\setm\{x_0\})\setm V$. 
Suppose that $V\neq\emptyset$ and let $G$ be a component of $V$. 
Then 
\[
	\Cp(X\setm G)
	\geq \Cp(X\setm V)
	\geq \Cp(X\setm(\Omega\cup 2B))
	> 0,
\]
and hence Lemma~4.3 in Bj\"orn--Bj\"orn~\cite{BB} 
(or Lemma~4.5 in \cite{BBbook}) 
implies that $\Cp(\bdy G)>0$. 
Let $B'$ be a sufficiently large ball so that 
$\Cp(B'\cap\bdy G)>0$. 
Since $\Cp(\{x_0\})=0$, 
it follows from 
the Kellogg property for bounded sets 
(Theorem~3.9 in Bj\"orn--Bj\"orn--Shan\-mu\-ga\-lin\-gam~\cite{BBS} 
or Theorem~10.5 in \cite{BBbook}) 
that there is a point 
$x_1\in(B'\cap\bdy G)\setm\{x_0\}$ that is regular with respect to 
$G':=G\cap B'$. 
As in \eqref{barrier-HP-eq} for $U$, 
we have $v|_{G'}=\uHp_{G'}v$, 
and it follows that
\[
	\lim_{G\ni y\to x_1}v(y)
	= \lim_{G'\ni y\to x_1}v(y)
	= \lim_{G'\ni y\to x_1}\uHp_{G'}v(y)
	= v(x_1)
	= h(x_1)
	< 0.
\]
Thus $v\not\equiv 0$ in $G$. 
As $v\leq 0$ is \p-harmonic in $G$ 
(by Theorem~4.4 in Hansevi~\cite{hansevi1}), 
it follows from the strong maximum principle 
(see 
Corollary~6.4 in Kinnunen--Shanmugalingam~\cite{KiSh01} 
or \cite[Theorem~8.13]{BBbook}), 
that $v<0$ in $G$ (and thus also in $V$). 
We conclude that   
$v<0$ in $(\Omega\cup 2B)\setm\{x_0\}$. 

Let $m=\sup_{\bdy B}v$. 
By compactness, we have $-r\leq m<0$. 
Since $v|_{(\Omega\cup 2B)\setm\itoverline{B}}$ 
is the continuous solution of the 
$\K_{h,v}((\Omega\cup 2B)\setm\itoverline{B})$-obstacle problem 
(by Proposition~3.7 in \cite{hansevi1}) 
and $h=-r$ in $(\Omega\cup 2B)\setm\itoverline{B}$, 
we see that 
$\sup_{(\Omega\cup 2B)\setm\itoverline{B}}v=m$. 
It follows that 
\[
	\limsup_{\Omega\ni y\to x}v(y)
	\leq m
	< 0
	\quad\text{for all }x\in\bdy\Omega\setm\itoverline{B}.
\]
Moreover, as $v$ is continuous in $2B$, 
it follows that 
\[
	\limsup_{\Omega\ni y\to x}v(y)
	= v(x)
	< 0
	\quad\text{for all }x\in(\bdy\Omega\cap\itoverline{B})\setm\{x_0\},
\]
and hence  
\[
	\limsup_{\Omega\ni y\to x}v(y)
	< 0
	\quad\text{for all }x\in\bdy\Omega\setm\{x_0\}.
\]

Since $v$ is bounded and superharmonic in $\Omega$, 
defining $w$ in the particular way on $\bdy\Omega$ 
as we did in \eqref{barrier-def-w} 
makes sure that $w\in\LL_w$,  
and hence $u\geq w$ in $\Omega$. 
It follows that $u$ is positive and satisfies 
condition~\ref{barrier-iii} in 
Definition~\ref{def:barrier}, as 
\[
	\liminf_{\Omega\ni y\to x}u(y)
	\geq \liminf_{\Omega\ni y\to x}(-v(y))
	= -\limsup_{\Omega\ni y\to x}v(y)
	> 0
	\quad\text{for all }x\in\bdy\Omega\setm\{x_0\}.
\]
Thus $u$ is a positive continuous barrier at $x_0$. 
\end{proof}

\section{The Kellogg property}
\label{sec:kellogg}
\begin{theorem}\label{thm:kellogg}
\textup{(The Kellogg property)}
If $I$ is the set of 
irregular points in $\bdyOmegaX$\textup{,} 
then $\Cp(I)=0$.
\end{theorem}
\begin{proof}
Cover $\bdyOmegaX$ by a countable set of balls 
$\{B_j\}_{j=1}^\infty$ and 
let $I_j=I\cap B_j$. 
Furthermore, let $I_j'$ 
be the set of irregular boundary points of $\Omega\cap B_j$, $j=1,2,\ldots$\,. 
Theorem~\ref{thm:barrier} 
(using that $\neg$\ref{barrier-reg-Om} $\imp$ $\neg$\ref{barrier-reg-B}) 
implies that 
$I_j\subset I_j'$. 
Moreover, 
$\Cp(I_j')=0$, by the Kellogg property for bounded sets 
(Theorem~3.9 in Bj\"orn--Bj\"orn--Shan\-mu\-ga\-lin\-gam~\cite{BBS} 
or Theorem~10.5 in \cite{BBbook}). 
Hence $\Cp(I_j)=0$ for all $j$, 
and thus by the subadditivity of the capacity, 
$\Cp(I)=0$. 
\end{proof}
As a consequence of Theorem~\ref{thm:kellogg} 
we obtain the following result, 
which in the bounded case is a direct consequence
of the results in 
Bj\"orn--Bj\"orn--Shan\-mu\-ga\-lin\-gam~\cite{BBS}, \cite{BBS2}.
\begin{theorem}\label{thm:uniq}
If $f\in C(\bdy\Omega)$\textup{,} 
then there exists a 
bounded \p-harmonic function $u$ on $\Omega$ 
such that there is a set $E\subset\bdyOmegaX$ with $\Cp(E)=0$ so that 
\begin{equation}\label{uniq-eq}
	\lim_{\Omega\ni y\to x}u(y)
	= f(x)
	\quad\text{for }x\in\bdy\Omega\setm(E\cup\{\infty\}).
\end{equation}
If moreover\textup{,} $\Omega$ is bounded or \p-parabolic\textup{,}
then $u$ is unique and 
$u=\Hp f$.
\end{theorem}
Existence holds also for \p-hyperbolic sets, 
which follows from the proof below, 
but uniqueness can fail. 
To see this, 
let $1<p<n$ and 
$\Omega=\R^n\setm\itoverline{B(0,1)}$ 
in unweighted $\R^n$. 
Then both 
$u(x)=|x|^{(p-n)/(p-1)}$ and $v\equiv 1$ are 
functions that are \p-harmonic in $\Omega$ 
and 
satisfy \eqref{uniq-eq} 
when $f \equiv 1$, 
with $E=\emptyset$.
\begin{proof}
Let $u=\uHp f$ 
and let $E$ be the set of irregular boundary points in $\bdyOmegaX$. 
Then $\Cp(E)=0$ 
by the Kellogg property (Theorem~\ref{thm:kellogg}), 
and $u$ is bounded, \p-harmonic, and satisfies \eqref{uniq-eq}, 
which shows the existence. 

For uniqueness, 
suppose that $\Omega$ is bounded or \p-parabolic, 
and that $u$ is a bounded \p-harmonic function 
that satisfies \eqref{uniq-eq}. 
By Lemma~5.2 in Bj\"orn--Bj\"orn--Shanmugalingam~\cite{BBSdir}, 
$\bCp(E,\Omega)\leq\Cp(E)$ 
(the proof is valid also if $\Omega$ is unbounded), 
and hence Corollary~7.9 in Hansevi~\cite{hansevi2} implies that 
$u=\Hp f$.
\end{proof}
Another consequence of the barrier characterization 
is the following restriction result.
\begin{proposition} \label{prop-subset}
Let $x_0\in\bdyOmegaX$ 
be regular\textup{,} and let $V\subset\Omega$ be open and 
such that $x_0\in\bdy V$.
Then $x_0$ is regular also with respect to $V$.
\end{proposition}
\begin{proof}
Using the barrier characterization the
proof of this fact is almost
identical to the proof of the implication
\ref{barrier-bar-cont-Om} $\imp$ \ref{barrier-bar-B}
in Theorem~\ref{thm:barrier}.
We leave the details to the reader.
\end{proof}

\section{Boundary regularity for obstacle problems}
\label{sec:bdy-reg-obst}
\begin{theorem}\label{thm:reg-obst-prob}
Let $\psi\colon\Omega\to\eR$ and $f\in\Dp(\Omega)$ 
be functions such that $\K_{\psi,f}\neq\emptyset$\textup{,} 
and let $u$ be the lsc-regularized solution 
of the $\K_{\psi,f}$-obstacle problem. 
If $x_0\in\bdyOmegaX$ is regular\textup{,} 
then 
\begin{equation}\label{reg-obst-prob-eq}
	m
	= \liminf_{\Omega\ni y\to x_0}u(y)
	\leq \limsup_{\Omega\ni y\to x_0}u(y)
	= M,
\end{equation}
where
\begin{align*}
	m
	&:= \sup\{k\in\R:(f-k)_\limminus\in\Dp_0(\Omega;B)
		\text{ for some ball }B\ni x_0\}, \\
	M
	&:= \max\Bigl\{M',\cplimsup_{\Omega\ni y\to x_0}\psi(y)\Bigr\}, \\
	M'
	&:= \inf\{k\in\R:(f-k)_\limplus\in\Dp_0(\Omega;B)
		\text{ for some ball }B\ni x_0\}.
\end{align*}
\end{theorem}
Roughly speaking, 
$m$ is the $\liminf$ of $f$ at $x_0$ in the Sobolev sense and 
$M'$ is the corresponding $\limsup$. 

Observe that it is not possible to replace $M$ by $M'$, 
since it can happen that $\cplimsup_{\Omega\ni y\to x_0}\psi(y)>M'$, 
see Example~5.7 in Bj\"orn--Bj\"orn~\cite{BB} 
(or Example~11.10 in \cite{BBbook}). 

In the case when $\Omega$ is bounded, 
this improves upon 
Theorem~5.6 in \cite{BB} (and Theorem~11.6 in \cite{BBbook}) 
in two ways: 
By allowing for $f\in\Dp(\Omega)$ 
and by having (two) equalities in \eqref{reg-obst-prob-eq}, 
instead of inequalities.
\begin{lemma}\label{lem:Dp0}
Assume that\/ $0<\tau<1$. 
If $h\in\Dp_0(\Omega;B)$ for some ball $B$\textup{,} 
then $h\in\Np_0(\Omega;\tau B)$. 
\end{lemma}
\begin{proof}
Let $h\in\Dp_0(\Omega;B)$ for some ball $B$. 
Extend $h$ to $B$ by letting $h$ be equal to zero in $B\setm\Omega$ 
so that $h\in\Dp(B)$. 
Theorem~4.14 in \cite{BBbook} 
implies that $h\in\Nploc(B)$, 
and hence $h\in\Np(\tau B)$. 
As $h=0$ in $\tau B\setm\Omega$, 
it follows that $h\in\Np_0(\Omega;\tau B)$. 
\end{proof}
It follows from Lemma~\ref{lem:Dp0} that the space $\Dp_0(\Omega;B)$ 
in the expressions for $m$ and $M'$ in Theorem~\ref{thm:reg-obst-prob} 
can in fact be replaced by the space $\Np_0(\Omega;B)$ 
without changing the values of $m$ and $M'$.
\begin{proof}[Proof of Theorem~\ref{thm:reg-obst-prob}]
Let $k>M$ 
be real and, using Lemma~\ref{lem:Dp0}, 
find a ball $B=B(x_0,r)$, 
with $r<\frac{1}{4}\diam X$, 
so that $(f-k)_\limplus\in\Np_0(\Omega;B)$ 
and $k\geq\cpesssup_{B\cap\Omega}\psi$. 
Let $V=B\cap\Omega$ and let 
\[
	v
	= \begin{cases}
		\max\{u,k\} & \text{in }V, \\
		k & \text{in }B\setm V.
	\end{cases}
\]

Since $0\leq(u-k)_\limplus\leq (u-f)_\limplus+(f-k)_\limplus\in\Np_0(\Omega;B)$, 
Lemma~5.3 in Bj\"orn--Bj\"orn~\cite{BB} 
(or Lemma~2.37 in \cite{BBbook}) shows that 
$(u-k)_\limplus\in\Np_0(\Omega;B)$. 
Because 
$\max\{u,k\}=k+(u-k)_\limplus$, 
we see that 
$(v-k)_\limplus\in\Np_0(V;B)$ and 
$v\in\Np(B)$. 
Let $U=\Omega\cap\frac{1}{3}B$. 
The boundary weak Harnack inequality 
(Lemma~5.5 in \cite{BB} or 
Lemma~11.4 in \cite{BBbook}) 
implies that $\oHp_{V}v$ is bounded from above on 
$\overline{U}$. 

By Lemma~4.7 in Hansevi~\cite{hansevi1}, 
it follows that 
\[
	\oHp_{V}v
	\geq \oHp_{V}k
	= k
	\geq \cpesssup_V\psi
	\quad\text{in }V,
\]
and hence 
$\oHp_{V}v$ 
is a solution of the 
$\K_{\psi,v}(V)$-obstacle problem. 
Furthermore, Proposition~3.7 in \cite{hansevi1} shows that 
$u$ is a solution of the $\K_{\psi,u}(V)$-obstacle problem, 
and thus 
$u\leq\oHp_{V}v$ 
in $V$, 
by Lemma~4.2 in \cite{hansevi1}. 
Hence $u$ 
is bounded from above on $\overline{U}$,  and thus $v$ 
is bounded on $\overline{U}$. 

By replacing $V$ by $U$ in the previous paragraph, 
we see that 
$u\leq\oHp_{U}v$ 
in $U$. 
It follows from Theorem~\ref{thm-hansevi2-main-res} 
(after multiplication by a suitable cutoff function)
that $\oHp_{U}v=\uHp_{U}v$. 
Theorem~\ref{thm:barrier} asserts that 
$x_0$ is regular also with respect to $U$. 
Hence, as $v\equiv k$ on $\tfrac{1}{3}B\cap\bdy U$, 
Theorem~\ref{thm:reg} shows that 
\[
	\limsup_{\Omega\ni y\to x_0}u(y)
	= \limsup_{U\ni y\to x_0}u(y)
	\leq \lim_{U\ni y\to x_0}\uHp_{U}v(y)
	= v(x_0)
	= k.
\]
Taking infimum over all $k>M$ shows that 
\begin{equation}\label{reg-obst-prob-eq-M}
	\limsup_{\Omega\ni y\to x_0}u(y)
	\leq M.
\end{equation}

Now let $k>\limsup_{\Omega\ni y\to x_0}u(y)$ be real. 
Then there is a ball $B\ni x_0$ 
such that $u\leq k$ in $B\cap\Omega$, 
and hence $(u-k)_\limplus\equiv 0$ in $B\cap\Omega$. 
It follows that 
\[
	0
	\leq (f-k)_\limplus
	\leq (f-u)_\limplus+(u-k)_\limplus\in\Dp_0(\Omega;B),
\]
and thus $(f-k)_\limplus\in\Dp_0(\Omega;B)$, 
by Lemma~2.8 in Hansevi~\cite{hansevi1}. 
This implies that $k\geq M'$, 
and hence taking infimum over all $k>\limsup_{\Omega\ni y\to x_0}u(y)$ 
shows that  
\begin{equation}\label{reg-obst-prob-eq-M'}
	\limsup_{\Omega\ni y\to x_0}u(y)
	\geq M'.
\end{equation}
We also know that $u\geq\psi$ q.e., 
so that 
\[
	\limsup_{\Omega\ni y\to x_0}u(y)
	\geq \cplimsup_{\Omega\ni y\to x_0}u(y)
	\geq \cplimsup_{\Omega\ni y\to x_0}\psi(y),
\]
which combined with \eqref{reg-obst-prob-eq-M} and 
\eqref{reg-obst-prob-eq-M'} shows that 
\[
	\limsup_{\Omega\ni y\to x_0}u(y)
	= M,
\]
and thus we have shown the last equality in 
\eqref{reg-obst-prob-eq}. 

To prove the other equality, 
let $k<\liminf_{\Omega\ni y\to x_0}u(y)$. 
Then there is a ball $B\ni x_0$ 
such that $k\leq u$ in $B\cap\Omega$, 
and hence $(k-u)_\limplus\equiv 0$ in $B\cap\Omega$. 
Lemma~2.8 in Hansevi~\cite{hansevi1} implies that 
$(f-k)_\limminus\in\Dp_0(\Omega;B)$, since 
\[
	0
	\leq (k-f)_\limplus
	\leq (k-u)_\limplus+(u-f)_\limplus\in\Dp_0(\Omega;B).
\]
Thus $k\leq m$, 
and hence taking supremum over all $k<\liminf_{\Omega\ni y\to x_0}u(y)$ 
shows that 
\[
	\liminf_{\Omega\ni y\to x_0}u(y)
	\leq m.
\]

We complete the proof by applying the first part of the proof to 
$h:=-f$ and $\psi\equiv-\infty$. 
Note that $\oHp h$ is the lsc-regularized solution of the 
$\K_{-\infty,-f}$-obstacle problem, 
and that $u\geq\oHp f=-\oHp h$, 
by Lemma~4.2 in Hansevi~\cite{hansevi1}. 
Let 
\[
	M''
	= \inf\{k\in\R:(h-k)_\limplus\in\Dp_0(\Omega;B)
		\text{ for some ball }B\ni x_0\}.
\] 
Then, as 
\begin{align*}
	\max\{M'', -\infty\} 
	&= \inf\{k\in\R:(f+k)_\limminus\in\Dp_0(\Omega;B) 
		\text{ for some ball }B\ni x_0\} \\
	&= -\sup\{k\in\R:(f-k)_\limminus\in\Dp_0(\Omega;B) 
		\text{ for some ball }B\ni x_0\} \\
	&= -m, 
\end{align*}
it follows that 
\[
	\liminf_{\Omega\ni y\to x_0}u(y)
	= -\limsup_{\Omega\ni y\to x_0}(-u)(y)
	\geq -\limsup_{\Omega\ni y\to x_0}\oHp h(y)
	= m.
	\qedhere
\]
\end{proof}
\begin{theorem}\label{thm:reg-obst-cont}
Let $\psi\colon\Omega\to\eR$ and $f\in\Dp(\Omega)$ 
be functions such that $\K_{\psi,f}\neq\emptyset$\textup{,} 
and let $u$ be the lsc-regularized solution of the 
$\K_{\psi,f}$-obstacle problem. 
Assume that $x_0\in\bdyOmegaX$ is regular and that either 
\begin{enumerate}
\item\label{reg-obst-cont-a}
$f(x_0):=\lim_{\Omega\ni y\to x_0}f(y)$ exists\textup{,} or 
\item\label{reg-obst-cont-b}
$f\in\Dp(\overline{\Omega}\cap B)$ for some ball $B\ni x_0$\textup{,} 
and $f|_{\bdy\Omega\cap B}$ 
is continuous at $x_0$. 
\end{enumerate}
Then $\lim_{\Omega\ni y\to x_0}u(y)=f(x_0)$ 
if and only if 
$f(x_0)\geq\cplimsup_{\Omega\ni y\to x_0}\psi(y)$.

In both cases we allow $f(x_0)$ to be $\pm\infty$. 
\end{theorem}
Note that it is possible to have 
$f(x_0)<\cplimsup_{\Omega\ni y\to x_0}\psi(y)$ 
and still have a solvable obstacle problem, 
see 
Example~5.7 in Bj\"orn--Bj\"orn~\cite{BB} 
(or Example~11.10 in \cite{BBbook}). 

The proof of Theorem~\ref{thm:reg-obst-cont} 
is similar to 
the proof of 
Theorem~5.1 in Bj\"orn--Bj\"orn~\cite{BB} 
(or Theorem~11.8 in \cite{BBbook}), 
but appealing to Theorem~\ref{thm:reg-obst-prob} above 
instead of 
Theorem~5.6 in \cite{BB} (or Theorem~11.6 in \cite{BBbook}). 
That one can allow for $f(x_0)=\pm\infty$ seems 
not to have been noticed before. 
\begin{proof}
Let $m$, $M$, and $M'$ be defined as in Theorem~\ref{thm:reg-obst-prob}. 
We first show that $M'\leq f(x_0)$. 
If $f(x_0)=\infty$ there is nothing to prove, so assume that 
$f(x_0)\in[-\infty,\infty)$ and let $\alpha>f(x_0)$ be real. 
Also let $B'=B(x_0,r)$ be chosen so that 
\[
	f(x)
	< \alpha
	\quad \text{for }
	\begin{cases}
		x\in B'\cap\Omega & 
			\text{in case \ref{reg-obst-cont-a}}, \\
		x\in B'\cap\bdy\Omega & 
			\text{in case \ref{reg-obst-cont-b}},
	\end{cases}
\]
with the additional requirement that
$B'\subset B$ in case \ref{reg-obst-cont-b}. 
Then $(f-\alpha)_\limplus\in\Dp_0(\Omega;B')$ and 
thus $M'\leq\alpha$. 
Letting $\alpha\to f(x_0)$ shows that $M'\leq f(x_0)$. 
Applying this to $-f$ yields $f(x_0)\leq m$. 

If $f(x_0)\geq\cplimsup_{\Omega\ni y\to x_0}\psi(y)$, 
then by Theorem~\ref{thm:reg-obst-prob}, 
\[
	f(x_0) 
	\leq  m
	= \liminf_{\Omega\ni y\to x_0}u(y)
	\leq \limsup_{\Omega\ni y\to x_0}u(y)
	= M
	\leq f(x_0),
\]
and hence 
$\lim_{\Omega\ni y\to x_0}u(y)=f(x_0)$. 

Conversely, if 
$f(x_0)<\cplimsup_{\Omega\ni y\to x_0}\psi(y)$, 
then, as $u\geq\psi$ q.e., 
we have 
\[
	f(x_0)
	< \cplimsup_{\Omega\ni y\to x_0}\psi(y)
	\leq \cplimsup_{\Omega\ni y\to x_0}u(y)
	\leq \limsup_{\Omega\ni y\to x_0}u(y).
	\qedhere
\]
\end{proof}
The following corollary is a special case of 
Theorem~\ref{thm:reg-obst-cont}. 
(For the existence of a continuous solution 
see Section~\ref{sec:p-harmonic}.)
\begin{corollary}\label{cor:obst-pt}
Let $f\in\Dp(\Omega)\cap C(\overline{\Omega})$ and let 
$u$ be the continuous solution of the 
$\K_{f,f}$-obstacle problem. 
If $x_0\in\bdyOmegaX$ is regular\textup{,} 
then $\lim_{\Omega\ni y\to x_0}u(y)=f(x_0)$.
\end{corollary}

\section{Additional regularity characterizations}
\label{sec:bdy-reg-further}
\begin{theorem}\label{thm:reg-2}
Let $x_0\in\bdyOmegaX$ and let $B$ be a ball such that $x_0\in B$. 
Then the following are equivalent\/\textup{:}
\begin{enumerate}
\item\label{reg-2-reg} 
The point $x_0$ is regular.
\item\label{reg-2-obst-var}
For all $f\in\Dp(\Omega)$ and all $\psi\colon\Omega\to\eR$ 
such that $\K_{\psi,f}\neq\emptyset$ and  
\[
	f(x_0)
	:= \lim_{\Omega\ni y\to x_0}f(y)
	\geq \cplimsup_{\Omega\ni y\to x_0}\psi(y)
\]
\textup{(}where the limit in the middle is assumed to exist
in $\eR$\/\textup{)}\textup{,}
the lsc-regularized solution $u$ of the 
$\K_{\psi,f}$-obstacle problem satisfies 
\[
	\lim_{\Omega\ni y\to x_0}u(y)
	= f(x_0).
\]
\item\label{reg-2-obst-from-3}
For all $f\in\Dp(\Omega\cup(B\cap\overline{\Omega}))$ and all 
$\psi\colon\Omega\to\eR$ 
such that 
$\K_{\psi,f}\neq\emptyset$\textup{,} 
$f|_{\bdy\Omega\cap B}$ is continuous at 
$x_0$ \textup{(}with $f(x_0)\in\eR$\textup{),} and 
\[
	f(x_0)
	\geq \cplimsup_{\Omega\ni y\to x_0}\psi(y),
\] 
the lsc-regularized solution $u$ of the 
$\K_{\psi,f}$-obstacle problem satisfies 
\[
	\lim_{\Omega\ni y\to x_0}u(y)
	= f(x_0).
\]
\item\label{reg-2-obst-dist}
The continuous solution $u$ of the 
$\K_{d_{x_0},d_{x_0}}$-obstacle problem\textup{,} 
where $d_{x_0}$ is defined by \eqref{eq-dx0}\textup{,} 
satisfies 
\begin{equation} \label{eq-obst-dx_0-u}
	\lim_{\Omega\ni y\to x_0}u(y)
	= 0.
\end{equation}
Moreover\textup{,} $u$ is a positive continuous barrier at $x_0$.
\end{enumerate}
\end{theorem}
\begin{proof}
\ref{reg-2-reg} $\imp$ \ref{reg-2-obst-var} and 
\ref{reg-2-reg} $\imp$ \ref{reg-2-obst-from-3} 
These implications follow from Theorem~\ref{thm:reg-obst-cont}. 

\smallskip

\ref{reg-2-obst-var} $\imp$ \ref{reg-2-obst-dist} 
and \ref{reg-2-obst-from-3} $\imp$ \ref{reg-2-obst-dist} 
That \eqref{eq-obst-dx_0-u} holds follows directly 
since \ref{reg-2-obst-var} or \ref{reg-2-obst-from-3} holds. 
Moreover, 
as $u\geq d_{x_0}$ everywhere in $\Omega$, 
we see that 
\[
	\liminf_{\Omega\ni y\to x}u(y)
	\geq d_{x_0}(x)
	> 0
	\quad\textup{for all }x\in\bdy\Omega\setm\{x_0\}.
\]
As $u$ is superharmonic (see Section~\ref{sec:p-harmonic}), 
it is a positive continuous barrier at $x_0$.

\smallskip

\ref{reg-2-obst-dist} $\imp$ \ref{reg-2-reg}
Since $u$ is a barrier at $x_0$, 
Theorem~\ref{thm:barrier} implies that $x_0$ is regular. 
\end{proof}
\begin{theorem}\label{thm:reg-3}
Let $x_0\in\bdyOmegaX$ and let $B$ be a ball such that $x_0\in B$. 
Then \ref{reg-3-reg} implies 
parts 
\ref{reg-3-cont-x0-var}--\ref{reg-3-superharm} below. 
Moreover\textup{,} 
if\/ $\Omega$ is bounded or \p-parabolic\textup{,} 
then parts \ref{reg-3-reg}--\ref{reg-3-superharm} 
are equivalent. 

\begin{enumerate}
\item\label{reg-3-reg} 
The point $x_0$ is regular. 
\item\label{reg-3-cont-x0-var} 
It is true that 
\[
	\lim_{\Omega\ni y\to x_0}\oHp f(y)
	= f(x_0)
\]
for all $f\in\Dp(\Omega)$ such that 
$f(x_0):=\lim_{\Omega\ni y\to x_0}f(y)$ exists. 
\item\label{reg-3-cont-x0} 
It is true that 
\[
	\lim_{\Omega\ni y\to x_0}\oHp f(y)
	= f(x_0)
\]
for all $f\in\Dp(\Omega\cup(B\cap\overline{\Omega}))$ 
such that $f|_{\bdy\Omega\cap B}$ is continuous at $x_0$. 
\item\label{reg-3-superharm} 
It is true that 
\[
	\liminf_{\Omega\ni y\to x_0}f(y)
	\geq f(x_0)
\]
for all $f\in\Dp(\Omega\cup(B\cap\overline{\Omega}))$ 
that are superharmonic in $\Omega$ 
and such that 
$f|_{\bdy\Omega}$ is lower semicontinuous at $x_0$. 
\end{enumerate}
\end{theorem}
As in Theorems~\ref{thm:reg-obst-cont} and~\ref{thm:reg-2} 
we allow for $f(x_0)=\pm\infty$ in 
\ref{reg-3-cont-x0-var}--\ref{reg-3-superharm}. 
We do not know if \ref{reg-3-reg}--\ref{reg-3-superharm} are 
equivalent when $\Omega$ is \p-hyperbolic.
\begin{proof}
\ref{reg-3-reg} $\imp$ \ref{reg-3-cont-x0-var} 
and \ref{reg-3-reg} $\imp$ \ref{reg-3-cont-x0} 
Apply Theorem~\ref{thm:reg-2} to $f$ (with $\psi\equiv-\infty$). 
Then these implications are immediate 
as $\oHp f$ is the continuous solution of the 
$\K_{-\infty,f}$-obstacle problem. 

\smallskip

\ref{reg-3-reg} $\imp$ \ref{reg-3-superharm} 
Theorem~\ref{thm:barrier} asserts that the point 
$x_0$ is regular with respect to $V:=\Omega\cap B$. 
If $f(x_0)=-\infty$ there is nothing to prove, so assume that 
$f(x_0)\in(-\infty,\infty]$ and let $\alpha<f(x_0)$ be real. 

As $f|_{\bdy\Omega}$ is lower semicontinuous at $x_0$, 
there is $r$ such that $0<r<\dist(x_0,\bdy B)$ and 
$f\geq\alpha$ in $B(x_0,r)\cap\bdy V$. 

Let $h=\min\{f,\alpha\}$, 
which is also superharmonic in $\Omega$, by Lemma~9.3 in \cite{BBbook}. 
It thus follows from Lemma~\ref{lem-super-obst-V} that 
$h$ is the lsc-regularized solution of the
$\K_{h,h}(V)$-obstacle problem. 
Since $h-\alpha=0$ in $B(x_0,r)\cap\bdy V$, 
we have 
\[
	h-\alpha \in\Dp_0(V;B(x_0,r)).
\]
By applying 
Theorem~\ref{thm:reg-obst-prob} 
with $h$ and $V$ in the place of $f=\psi$ and $\Omega$, respectively, 
we see that $m\geq\alpha$, 
where $m$ is as in Theorem~\ref{thm:reg-obst-prob}, 
and hence 
\[
	\liminf_{\Omega\ni y\to x_0}f(y)
	= \liminf_{V\ni y\to x_0}f(y)
	\geq \liminf_{V\ni y\to x_0}h(y)
	= m
	\geq\alpha
	.
\]
Letting $\alpha\to f(x_0)$ yields the desired result.

\bigskip

We now assume that $\Omega$ is bounded or \p-parabolic. 

\smallskip

\ref{reg-3-cont-x0-var} $\imp$ \ref{reg-3-reg} and 
\ref{reg-3-cont-x0} $\imp$ \ref{reg-3-reg} 
Observe that the function $d_{x_0}$ in Theorem~\ref{thm:reg} 
satisfies the conditions for $f$ in both 
\ref{reg-3-cont-x0-var} and \ref{reg-3-cont-x0}. 
Theorem~\ref{thm-hansevi2-main-res} 
applies to $d_{x_0}$, 
and hence it follows that $x_0$ is regular, 
by Theorem~\ref{thm:reg}, as 
\[
	\lim_{\Omega\ni y\to x_0}\uHp d_{x_0}(y)
	= \lim_{\Omega\ni y\to x_0}\oHp d_{x_0}(y)
	= d_{x_0}(x_0)
	= 0.
\]

\smallskip

\ref{reg-3-superharm} $\imp$ \ref{reg-3-reg} 
Let 
\[
	f
	= \begin{cases}
		\oHp d_{x_0} & \textup{in }\Omega, \\
		d_{x_0} & \textup{on }\bdy\Omega.
	\end{cases}
\]
Because both $f$ and $-f$ satisfy the hypothesis in \ref{reg-3-superharm}, 
we have 
\[
	0
	= f(x_0)
	\leq \liminf_{\Omega\ni y\to x_0}f(y)
	= \liminf_{\Omega\ni y\to x_0}\oHp d_{x_0}(y)
\]
and
\[
	\limsup_{\Omega\ni y\to x_0}\oHp d_{x_0}(y)
	= -\liminf_{\Omega\ni y\to x_0}(-f(y))
	\leq f(x_0)
	= 0.
\]
Theorem~\ref{thm-hansevi2-main-res} 
implies that 
$\oHp d_{x_0}=\uHp d_{x_0}$, 
and hence 
\[
	0
	\leq \liminf_{\Omega\ni y\to x_0}\uHp d_{x_0}(y)
	\leq \limsup_{\Omega\ni y\to x_0}\uHp d_{x_0}(y)
	\leq 0,
\]
which shows that $\lim_{\Omega\ni y\to x_0}\uHp d_{x_0}(y)=0$. 
Thus $x_0$ is regular by Theorem~\ref{thm:reg}.
\end{proof}
The following two results remove the assumption of bounded sets 
from the \p-harmonic versions of 
Lemma~7.4 and Theorem~7.5 in Bj\"orn~\cite{ABcluster} 
(or Theorem~11.27  and Lemma~11.32 in \cite{BBbook}).
\begin{theorem}\label{thm:components-gen}
If $x_0\in\bdyOmegaX$ is irregular with respect to $\Omega$\textup{,} 
then there is exactly one component $G$ of\/ $\Omega$ with $x_0\in\bdy G$ 
such that $x_0$ is irregular with respect to $G$.
\end{theorem}
\begin{lemma}\label{lem:disjoint}
Suppose that\/ $\Omega_1$ and\/ $\Omega_2$ are 
nonempty disjoint open subsets of $X$. 
If $x_0\in(\bdy\Omega_1\cap\bdy\Omega_2)\setm\{\infty\}$\textup{,} 
then $x_0$ is regular with respect to 
at least one of these sets. 
\end{lemma}
The lemma follows directly from 
the sufficiency part of the Wiener criterion, 
see \cite{ABcluster} or \cite{BBbook}. 
With straightforward 
modifications of the proof of 
Theorem~7.5 in \cite{ABcluster} 
(or Theorem~11.27 in \cite{BBbook}), 
we obtain a proof for Theorem~\ref{thm:components-gen}. 
For the reader's convenience, 
we give the proof here.
\begin{proof}[Proof of Theorem~\ref{thm:components-gen}]
Suppose that $x_0\in\bdyOmegaX$ is irregular. 
Then Theorem~\ref{thm:reg} implies that 
\[
	\limsup_{\Omega\ni y\to x_0}\uHp d_{x_0}(y)
	> 0.
\]
Let $\{y_j\}_{j=1}^\infty$ be a sequence in $\Omega$ 
such that 
\[
    \lim_{j\to\infty}y_j 
    = x_0
	\quad\text{and}\quad
	\lim_{j\to\infty}\uHp d_{x_0}(y_j)
	= \limsup_{\Omega\ni y\to x_0}\uHp d_{x_0}(y).
\]

Assume that there are infinitely many components of $\Omega$ 
containing points from the sequence $\{y_j\}_{j=1}^\infty$. 
Then we can find a subsequence $\{y_{j_k}\}_{k=1}^\infty$ 
such that each component of $\Omega$ contains at most one 
point from the sequence $\{y_{j_k}\}_{k=1}^\infty$. 
Let $G_k$ be the component of $\Omega$ containing $y_{j_k}$, $k=1,2,\ldots$\,. 
Then 
\[
	\lim_{k\to\infty}\uHp d_{x_0}(y_{j_{2k}})
	= \lim_{k\to\infty}\uHp d_{x_0}(y_{j_{2k+1}})
	> 0,
\]
and thus $x_0$ is irregular both with respect to 
$\Omega_1:=\bigcup_{k=1}^\infty G_{2k}$ and with respect to 
$\Omega_2:=\bigcup_{k=1}^\infty G_{2k+1}$, 
by Theorem~\ref{thm:reg}. 
Since $\Omega_1$ and $\Omega_2$ are disjoint, 
this contradicts Lemma~\ref{lem:disjoint}. 
We conclude that there are only finitely many components of $\Omega$ 
containing points from the sequence $\{y_j\}_{j=1}^\infty$. 

Thus there is a component $G$ 
that contains infinitely many of the points 
from the sequence $\{y_j\}_{j=1}^\infty$. 
So there is a subsequence $\{y_{j_k}\}_{k=1}^\infty$ 
such that $y_{j_k}\in G$ for every $k=1,2,\ldots$\,. 
It follows that $x_0\in\bdy G$ and as 
\[
	\lim_{k\to\infty}\uHp d_{x_0}(y_{j_{k}})
	> 0,
\]
$x_0$ must be irregular with respect to $G$. 

Finally, if $G'$ is any other component of $\Omega$ with $x_0\in\bdy G'$, 
then, by Lemma~\ref{lem:disjoint}, $x_0$ is regular with respect to $G'$.
\end{proof}


\end{document}